\documentclass[12pt, twoside, draft]{article} 
  
\usepackage{a4wide}
\usepackage{array}

\usepackage{amssymb}
\usepackage{amsmath}
\usepackage{amsthm}
\usepackage{mathrsfs}

\usepackage{tikz}
\usetikzlibrary{calc}

\usepackage{caption}
\usepackage{mathtools}
\usepackage{xypic} 

\usepackage{fancyhdr}
\newcounter{roem} 
\newenvironment{enums}{\begin{list}{\normalfont (\roman{roem})}{\usecounter{roem}%
								\itemsep0mm \topsep1mm
                \labelwidth7mm \leftmargin10mm \labelsep3mm }}{\end{list}}

\theoremstyle{plain}
\newtheorem{thm}{Theorem}
\newtheorem{lem}[thm]{Lemma}

\newtheorem{beh}[thm]{Conjecture}

\newtheorem*{lem*}{Lemma}
\newtheorem*{thm*}{Theorem}
\newtheorem*{satz*}{Satz}

\theoremstyle{definition}
\newtheorem{defn}[thm]{Definition}

\def\eqref#1{(\ref{#1})}

\def\comp{\mathrm{Comp}}

\def\stab{\mathrm{Stab}}

\def\z1{\mathrm{Z1}}

\renewcommand{\stab}{\mathrm{Stab}}

\newcommand{\id}{\mathrm{id}}

\newcommand{\im}{\operatorname{Im}}
\def\endo{\operatorname{End}}
\renewcommand{\hom}{\operatorname{Hom}}

\newcommand{\ct}{\Gamma}
\newcommand{\cak}[1]{\Gamma_{#1}}
\newcommand{\cc}{\operatorname{cc}}

\newcommand{\vertex}[1]{\mathrm{V}(#1)}
\newcommand{\polycyclic}[1]{\mathrm{Pc}}

\newcommand{\aut}{\operatorname{Aut}}

\def\rho{\varrho}

\begin{document}

\title{Corrigendum to ``On the classification of prime-power groups by coclass``}
\author{Martin Couson}

\maketitle

\begin{abstract}
A further step towards a possible classification of $p$-groups by coclass was made by Eick \& Leedham-Green \cite{Eick2007} in 2007. They gave a constructive proof of the ultimate periodicity of shaved coclass graphs. By giving a counterexample we show that there is a mistake in \cite{Eick2007}, and we fill this gap.
\end{abstract}

\section{Introduction}
In 2000 du Sautoy \cite{Sautoy2000} has proven that every shaved coclass tree is ultimately periodic. Hereby he gave a nonconstructive proof by using zeta functions. In contrast to this Eick \& Leedham-Green used homological algebra in \cite{Eick2007} to confirm this result and to state structure theorems for $p$-groups. These structure theorems yield several results for invariants of $p$-groups, for instance almost all finite $2$-groups of fixed coclass satisfy the divisibility conjecture \cite{Eick2006}.

In this paper we show that there is a mistake in \cite{Eick2007} by giving a counterexample, we fill this gap and thus complete the work in \cite{Eick2007}. First let us state the main results of~\cite{Eick2007} and recall some notation.

\section*{Coclass graph}
Let $p$ be a prime. The coclass of a finite $p$-group $G$ of order $p^n$ and nilpotency class $c$ is defined as $\cc G := n - c$. The isomorphism classes of the $p$-groups of fixed coclass $r \in \mathbb{N}$ form the vertices of the so-called coclass graph $\mathcal{G}(p,r)$. Further, there is a directed edge from a group $G$ to another group $H$ in $\mathcal{G}(p,r)$ if there exists a normal subgroup $N$ of $H$ such that $H/N \cong G$. The coclass graphs are far away from being fully understood, but there exist some very interesting results about them \cite[Section 2]{Eick2007}: for example there are only finitely many infinite paths in $\mathcal{G}(p,r)$ up to equivalence. These paths arise from pro-$p$-groups of coclass $r$; hereby we say that a pro-$p$-group $S$ with lower central series
$$
S = S_1 \geq S_2 \geq S_3 \geq \cdots
$$
has coclass $r$ if every quotient $S/S_i$ is finite and $\lim \cc S/S_i = r$. For each isomorphism class of a pro-$p$-group $S$ of coclass $r$ we can define a subgraph $\mathcal{T}(S)$ in $\mathcal{G}(p,r)$ called coclass tree: let $t = t(S)$ be the minimal natural number such that $S/S_t$ is not isomorphic to a quotient of a pro-$p$-group $R$ of coclass $r$ with $R \not \cong S$. Note that such a number $t(S)$ exists, since up to isormorphism there exists only finite manly pro-$p$-groups of coclass $r$ by the proven Coclass Conjecture D. Then $\mathcal{T}(S)$ is defined as the descendant tree of $S/S_t$.

The coclass trees could be considered as the essence of the coclass graph: they correspond 1-1 to the equivalence classes of infinite paths in $\mathcal{G}(p,r)$ and they form a partition of all but finitely many vertices of $\mathcal{G}(p,r)$.

Hence it is essential to understand coclass trees in order to get a deeper insight into $p$-groups.

\section*{Shaved coclass tree}

In general coclass trees are not ultimately periodic as a direct consequence of Coclass Conjecture D and \cite[Remark 4]{Eick2007}. But shaved coclass trees are. In particular for a shaved coclass tree of periodicity $m$, say, and for $k,l \gg 0$ with $k = l \mod m$ there exists a graph isomorphism between the $k$-th and $l$-th branch of the shaved coclass tree. These graph isomorphisms are described explicitly by Eick \& Leedham-Green~\cite{Eick2007}. This description requires a deep analysis of the structure of the groups in a shaved coclass graph.

For a natural number $k$ and a coclass tree $\mathcal{T}(S)$, the shaved coclass tree $\mathcal{T}_k(S)$ is defined as the subgraph of $\mathcal{T}(S)$ induced by the vertices $G \in \mathcal{T}(S)$ with distance at most $k$ from some factor group $S/S_i$ with $i \geq t(S)$; that is, there exists a normal subgroup $N$ of $G$ of order at most $p^k$ with $G/N \cong S/S_i$. Evidently, every group $G$ of $\mathcal{T}(S)$ lies in $\mathcal{T}_k(S)$ for some $k \in \mathbb{N}$. If $p$ equals $2$, then there exists even a natural number $k$ such that $\mathcal{T}_k(S) = \mathcal{T}(S)$.

Now we fix $k \in \mathbb{N}$. Eick \& Leedham-Green showed the existence of a natural number $l = l(k,S)$, a free $\mathbb{Z}_p$-module $T$ and a finite $p$-group $R$ acting uniserially on $T$ such that every descendant of $S/S_{l}$ in $\mathcal{T}_k(S)$ is a group extension of a quotient of $T$ by $R$. In order to recall the definition of uniserial action we introduce the following notation.

For a group $G$ and a $G$-module $A$ we write $a.g$ to denote the image of $a \in A$ under $g \in G$. Further, let $\mathbb{Z}_p$ denote the $p$-adic integers and let the rank of a $\mathbb{Z}_p$-module be defined as the rank of its Frattini quotient. Note that the Frattini quotient of a $\mathbb{Z}_p$-module $N$ is $N/pN$ and hence elementary abelian. 
\begin{defn}[cf. {\cite[Definition 4.1.1]{Leedham-Green2002}}]
Let $G$ be a finite $p$-group, acting on a $\mathbb{Z}_p$-module $N$ of finite rank $d$.
\begin{enums}
\item The \emph{lower $G$-central series of $N$} is defined as 
$$ 
N = N_0 \geq N_1 \geq N_2 \geq \cdots,
$$
where $N_0 := N$ and $N_{i+1}:= N_{i+1}(G) := [G, N_i] = \langle n.g -n \mid n \in N_i, g \in G\rangle$ for each $i \geq 0$.
\item $G$ acts \emph{uniserially} \index{uniserial action} on $N$, if $[N_i : N_{i+1}] = p$ for every $i$ with $N_i \neq 0$.
\end{enums}
\end{defn}

\begin{thm}[{\cite[Theorem 7]{Eick2007}}] \label{thmdescendantscctree}
There exists a natural number $l = l(k,S)$ such that 
\begin{enums}
 \item the factor group $R := S/S_{l}$ lies in $\mathcal{T}_k(S)$, the group $T := S_{l}$ is a free $\mathbb{Z}_p$-module of finite rank $d \in \mathbb{N}_0$ on which $R$ acts uniserially with $R$-central series $T = T_0 > T_1 > \cdots$, say, and
 \item every descendant $G \in \mathcal{T}_k(S)$ of $R$ is isomorphic to a group extensions of the $R$-module $T/T_n$ by $R$, where $n \in \mathbb{N}_0$ is determined by $|G| = |R| \cdot p^n$.
\end{enums}
\end{thm}

By Theorem \ref{thmdescendantscctree} all but finitely many groups of $\mathcal{T}_k(s)$ are group extensions of $T/T_n$ by $R$ for suitable $n$. In \cite{Eick2007} Eick \& Leedham-Green identified the elements in $H^2(R,T/T_n)$ with $n \in \mathbb{N}_0$ which describe these extensions. Further, they showed that the cohomology groups $H^2(R, T/T_n)$ with $n \in \mathbb{N}_0$ lie in finitely many isomorphism classes, and they gave explicit group isomorphisms between these cohomology groups. This led to explicit graph isomorphisms between branches of the shaved coclass tree yielding the ultimate periodicity.

Recall that the uniserial action of $R$ on the free $\mathbb{Z}_p$-module $T \cong \mathbb{Z}_p^d$ yields $T_{n+d} = p \cdot T_n$. The $R$-module isomorphism $T_n \rightarrow T_{n+d}$, $t \mapsto p\cdot t$ induces a group isomorphism $\mu := \mu(n) : H^3(R, T_n) \rightarrow H^3(R, T_{n+d})$ of cohomology groups.

\begin{thm}[{\cite[Theorem 8]{Eick2007}}] \label{thmseccohom}
There exists a natural number $m = m(l, S)$ such that for every $n \geq m$ 
\begin{enums}
 \item the cohomology group $H^2(R, T/T_n)$ is canonically isomorphic to the direct product $H^2(R, T) \oplus H^3(R, T_n)$,
 \item the map $(\id \oplus \mu) : H^2(R, T/T_n) \rightarrow H^2(R, T/T_{n+d})$ defined by 
$$
H^2(R, T) \oplus H^3(R, T_n) \rightarrow H^2(R, T) \oplus H^3(R, T_{n+d}), \ (\alpha, \beta) \mapsto (\alpha, \mu(\beta))
$$ 
and by the canonical isomorphisms of (i) is a group isomorphism.
\end{enums}
\end{thm}

We write $\mathcal{B}_i$ to denote the $i$-th branch of the coclass tree $\mathcal{T}_k(S)$, that is, $\mathcal{B}_i$ is the subgraph induced by the descendants of $S/S_i$ which are not descendants of $S/S_{i+1}$. Further, for a group $G \in \mathcal{B}_i$ of order $|R| \cdot p^n$ let $\gamma_G$ denote an element of $H^2(R, T/T_n)$ such that the group extension $E(\gamma_G)$ defined by $\gamma_G$ is isomorphic to $G$. The main result of \cite{Eick2007} is the following.

\begin{thm}[{\cite[Theorem 9]{Eick2007}}] \label{thmgraphiso}
There exists an integer $f = f(k,S)$ such that, if $i \geq f$, then the map $\pi: \mathcal{B}_i \mapsto \mathcal{B}_{i+d}: G \mapsto E( (\id \oplus \mu)(\gamma_G))$ induces a graph isomorphism (independently of the choice of the element $\gamma_G$ defining G).
\end{thm}

One major step in proving Theorem \ref{thmgraphiso} is to show that the map $(\id \oplus \mu)$ induces a 1-1~correspondence between orbits of cohomology groups under the action of compatible pairs. The above mentioned mistake occurred at this stage of the proof.

\section*{Compatible pairs}

We recall the concept of compatible pairs. For a $R$-module $V$ let $\overline{\cdot}$ denote the homomorphism $R \rightarrow \aut V$, $g \mapsto \overline{g}$ induced by the group action of $R$ on $V$. Then the set
$$
\comp(R, V) := \{ (\beta, \epsilon) \in \aut R \times \aut V \mid \overline{g^{\beta}} = \overline{g}^{\epsilon} \ \text{for all } g \in R \}
$$
is a subgroup of the direct product $\aut R \times \aut V$. The elements of $\comp(R, V)$ are called compatible pairs of $R$ and $V$. The automorphism group $\aut V$ acts naturally on the abelian group $V$ via $t.\epsilon := \epsilon(t)$ for $t\in V$ and $\epsilon \in \aut V$. Further, the group action of $\comp(R, V)$ on $Z^2(R, V)$ is defined by $\gamma^{(\beta, \epsilon)}(g, h) := \gamma(g^{\beta^{-1}}, h^{\beta^{-1}}).\epsilon$ for $\gamma \in Z^2(R, V)$ and $(\beta, \epsilon) \in \comp(R, V)$. This induces an action of $\comp(R, V)$ on $H^2(R, V)$. 

If the coclass of an extension $E(\gamma)$ with $\gamma \in H^2(R,V)$ coincides with the coclass of~$R$, then the embedding of $V$ in $E(\gamma)$ is a characteristic subgroup. This yields the following result.

\begin{thm}[{\cite[Theorem 20]{Eick2007}}] \label{thmstrongiso}
Let $V$ be a $R$-module. Let $\gamma$, $\delta \in H^2(R, V)$ be such that their corresponding extensions $E(\gamma)$ and $E(\delta)$ of $V$ by $R$ have coclass $\cc R$. Then $E(\gamma)$ is isomorphic to $E(\delta)$ if and only if $\gamma$ and $\delta$ lie in the same orbit under the action of $\comp(R, V)$.
\end{thm}

\section*{The gap in \cite{Eick2007}}

Let us come back to our initial problem: the proof of Theorem \ref{thmgraphiso} {\cite[Theorem 9]{Eick2007}}. The graph isomorphism given in Theorem \ref{thmgraphiso} is defined by group isomorphisms $(\id \oplus \mu)$ of second cohomology groups. In order to show that $(\id \oplus \mu)$ is well-defined we should make sure that isomorphic groups are mapped to isomorphic groups under $(\id \oplus \mu): H^2(R, T/T_n) \rightarrow H^2(R,T/T_{n+d})$. By Theorem \ref{thmstrongiso} it is sufficient to prove the following.

\begin{thm} \label{thmbijorbits}
For every sufficiently large $n$ the group isomorphism $(\id \oplus \mu)$ induces an $1$-$1$ correspondence between the orbits of $H^2(R, T/T_n)$ and $H^2(R, T/T_{n+d})$ under the action of $\comp(R, T/T_n)$ and $\comp(R, T/T_{n+d})$, respectively. 
\end{thm}

We prove Theorem \ref{thmbijorbits} in Section \ref{secproofofthmbijorbits}. The original proof of Theorem \ref{thmbijorbits} given by Eick \& Leedham-Green in \cite{Eick2007} is based on \cite[Theorem 23]{Eick2007}. According to this theorem the $\comp(R, T/T_n)$-module structure of $H^2(R,T/T_n)$ coincides with its $\comp(R, T)$-module structure. The second cohomology group $H^2(R, T/T_n)$ could be considered as $\comp(R,T)$-module due to the following result.

\begin{thm}[{\cite[Lemma 21]{Eick2007}}]
Let $V$ be a uniserial $R$-module with lower $R$-central series $V = V_0 \geq V_1 \geq V_2 \geq \cdots$. If $(\nu, \tau) \in \comp(R, V)$ and $n \in \mathbb{N}$, then $V_n$ is invariant under $\tau$ and hence $(\nu, \tau)$ induces compatible pairs in $\comp(R, V_n)$ and $\comp(R, V/V_n)$.
\end{thm}

Now we are able to state the false Theorem of \cite{Eick2007}.

\begin{beh} [{\cite[Theorem 23]{Eick2007}}] \label{conjtheorem23}
For every $n \gg 0$ the $\comp(R, T/T_n)$-module structure of $H^2(R,T/T_n)$ coincides with its $\comp(R, T)$-module structure, that is, for every compatible pair $(\beta, \epsilon) \in \comp(R, T/T_n)$ there exists a compatible pair $(\beta, \delta) \in \comp(R, T)$ inducing the same automorphism of $H^2(R, T/T_n)$.
\end{beh}

A counterexample to Conjecture \ref{conjtheorem23} is given in Section \ref{seccounterexample}. First we give a corrected proof of Theorem \ref{thmbijorbits}.

\section{Proof of Theorem \ref{thmbijorbits}} \label{secproofofthmbijorbits}

As a first step we analyze the compatible pairs of $R$ and $T/T_n$. To ease the notation put
$$
A_n := T/T_n, \ \Gamma := \comp(R, T) \text{ and } \Gamma_n := \comp(R, A_n).
$$

The second component of each compatible pair $(\beta, \epsilon)$ is a homomorphisms of $R$-modules due to the following: For $\mathbb{Z}_p R$-modules $V$ and $W$ we write $\hom_R(V, W)$ to denote the group of $\mathbb{Z}_p R$-module homomorphisms $V \rightarrow W$, and we put $\endo_R V := \hom_R(V, V)$. Since $R$ acts on $T \cong \mathbb{Z}_p^{d}$, the group $T$ can be considered as $\mathbb{Z}_p R$-module. For an automorphism $\beta \in \aut R$ let $T_n^{(\beta)}$ denote the $R$-module, which equals $T_n$ as $\mathbb{Z}_p$-module and on which $R$ acts via $t*g := t.(g^\beta)$. Further, put $A_n^{(\beta)} := T^{(\beta)}/ T_n^{(\beta)}$, $\endo_R^{\beta} T := \hom_R( T, T^{(\beta)} )$ and $\endo_R^{\beta} A_n := \hom_R( A_n, A_n^{(\beta)} )$. Then we have
\begin{align}
  \Gamma          & = \{ (\beta, \epsilon) \in \aut R \times \aut T \mid \epsilon \in \endo_R^{\beta} T \} \text{ and} \label{eqcomp} \\
  \Gamma_n        & = \{ (\beta, \epsilon) \in \aut R \times \aut A_n \mid \epsilon \in \endo_R^{\beta} A_n \} \label{eqcompak}.  
\end{align}

In order to describe the compatible pairs it is evidently necessary to understand $R$-homomorphisms.

For an $R$-module $V$, an element $v \in V$ and a subgroup $H$ of $R$ let $R_v$ and $C_V(H)$ denote the stabilizer of $v$ in $R$ and the $\mathbb{Z}_p$-module of elements in $V$ fixed by $H$, respectively, that is,
$$
R_v := \stab_R(v) := \{ g \in R \mid v.g = v \} \text{ and } C_V(H) := \{ v \in V \mid v.g = v \text{ for } g \in H \}.
$$

\begin{lem} \label{lemendo}
Let $V$ and $W$ be $\mathbb{Z}_p$-modules, and assume that $R$ acts uniserially on $V$ with lower $R$-central series $V = V_0 \geq V_1 \geq V_2 \geq \cdots$, say. Further, let $v_0$ be an element of $V$ with $\langle v_0, V_1 \rangle = V$, and define $\eta: \hom_R( V, W ) \rightarrow W$, $\varphi \mapsto \varphi(v_0)$. Then $\eta$ is a $\mathbb{Z}_p$-module monomorphism with $\im \eta = C_W( R_{v_0})$. 
\end{lem}

\begin{proof}
Let $f$ denote the rank of $V$. Since $R$ acts uniserially on $V$, there exist elements $v_1, \ldots, v_{f-1}$ of $V$ such that $v_i \in [R, \langle v_0, \ldots, v_{i-1} \rangle]$ and $\langle v_i, V_{i+1} \rangle = V_i$ for $1 \leq i \leq f-1$. Note that the image of $v_i$ under an $R$-homomorphism $\varphi$ is uniquely determined by $\varphi(v_0)$. By the uniserial action of $R$ on $V$, we have $\bigoplus_{i=0}^{f-1} \mathbb{Z}_p v_i = V$. Hence $\eta$ is a $\mathbb{Z}_p$-module monomorphism. It remains to show that $\im \eta = C_W( R_{v_0} )$. Let $\varphi \in \hom_R(V, W)$. Then for every $g \in R_{v_0}$ we have $\varphi(v_0).g = \varphi(v_0.g) = \varphi(v_0)$ and thus the element $\varphi(v_0) = \eta(\varphi)$ lies in $C_W( R_{v_0} )$. Conversely, let $w$ be an element of $C_W( R_{v_0} )$ and define $\varphi: V \rightarrow W$ via $\varphi(v_0) = w$ and $\varphi( v_0.\sum_{g\in R} a_g g) := \sum_{g \in R} a_g w.g$ for $\sum_{g\in R} a_g g \in \mathbb{Z}_p R$. We show that $\varphi$ is well-defined. Let $g, g' \in R$ be with $v_0.g = v_0.g'$, in particular $g' g^{-1} \in R_{v_0}$. Since $w \in C_W( R_{v_0} )$, we have $w.g'g^{-1} = w$ and thus $\varphi( v_0. g') = w.g' = (w.g'g^{-1}).g = w.g = \varphi( v_0.g)$. Thus $\varphi$ is well-defined and the result follows.
\end{proof} 

Let $t_0$ be an element of $T$ with $\langle t_0, T_1 \rangle = T$. For $\beta \in \aut R$ Lemma \ref{lemendo} yields $\endo_R^{\beta} A_n \cong C_{A_n^{\beta}}( R_{t_0 + T_n} )$. Recall that the $0$-th cohomology group $H^0( R_{t_0 + T_n}, A_n^{\beta} )$ is defined as $C_{A_n^{\beta}}( R_{t_0 + T_n} )$. Generalizing  \cite[Theorems 16 \& 18]{Eick2007} we show in Section~\ref{subseccohomtfree} that $H^0(R_{t_0 + T_n}, A_n^{(\beta)})$ is isomorphic to the direct sum of $H^1(R_{t_0 + T_n}, T_n^{(\beta)})$ and a quotient of $H^0(R_{t_0 + T_n}, T^{(\beta)})$ for every sufficiently large $n$.

\subsection{Cohomology groups of quotients of torsion-free modules} \label{subseccohomtfree}
Let $G$ and $\mathbb{Z}_p$ denote a finite group and the $p$-adic integers, respectively. Further, let
$$
0 \rightarrow A \stackrel{\iota}{\longrightarrow} B \stackrel{\vartheta}{\longrightarrow} B/A \rightarrow 0
$$
be a short exact sequence of $\mathbb{Z}_p G$-modules, where $\iota$ is an embedding $A \hookrightarrow B$ and $\vartheta$ is the natural homomorphism $B \twoheadrightarrow B/A$. This section deals with the decomposition of the $m$-th cohomology group $H^m(G, B/A)$ into a direct sum of $H^m(G,B)$ and $H^{m+1}(G,A)$ for $m \in \mathbb{N}_0$ and suitable modules $A$ and $B$.

The sequence $H^m(G, A) \rightarrow H^m(G, B) \rightarrow H^m(G, B/A)$ induced by $\iota$ and $\vartheta$ is exact by the Snake Lemma. Further, the Snake Lemma yields the existence of a connecting homomorphism $\delta_m: H^m(G, B/A) \rightarrow H^{m+1}(G, A) $ making the sequence
{\footnotesize
\begin{align} \label{escohoms}
\xymatrix@C=1.3em{
H^m(G, A) \ar[r] & H^m(G, B) \ar[r] & H^m(G, B/A) \ar[r]^{\delta_m} & H^{m+1}(G, A) \ar[r] & H^{m+1}(G, B) \ar[r] & H^{m+1}(G, B/A)
}
\end{align}
}exact. The Exact Sequences \eqref{escohoms} can be combined to the well-known long exact sequence 
{\small
\begin{align*}
{
0 \rightarrow H^0(G, A) \rightarrow \cdots \rightarrow H^m(G, A) \rightarrow H^m(G, B) \rightarrow H^m(G,B/A) \rightarrow H^{m+1}(G, A) \rightarrow \cdots .
}
\end{align*}}

For further details we refer to \cite[page 197]{Leedham-Green2002}. Now, we fix $m \in \mathbb{N}_0$ and we put
$$
e := \exp H^m(G, B) \text{ and } f := \exp H^{m+1}(G, A),
$$
where $\exp U$ denotes the exponent of a group $U$. The following lemma generalizes \cite[Theorem 16]{Eick2007} and its proof follows the original proof.

\begin{lem} \label{lemescohoms}
Assume that $B$ is a free $\mathbb{Z}_p$-module of finite rank and $A$ is a submodule of $f \cdot B$. Then the Exact Sequence \eqref{escohoms} gives rise to an exact sequence
\begin{align} \label{escohom}
 H^m(G, B) \longrightarrow H^m(G, B/A) \longrightarrow H^{m+1}(G, A) \rightarrow 0.
\end{align}
If additionally $A \leq e \cdot B$, then the Exact Sequence \eqref{escohoms} leads to a short exact sequence
\begin{align} \label{sescohom}
 0 \rightarrow H^m(G, B) \longrightarrow H^m(G, B/A) \longrightarrow H^{m+1}(G, A) \rightarrow 0.
\end{align}
\end{lem}

\begin{proof}
For $l \in \mathbb{N}_0$ let $\iota_l$ denote the group homomorphism $\iota_l: H^l(G, A) \rightarrow H^l(G, B)$ induced by $A \hookrightarrow B$. In order to prove Lemma \ref{lemescohoms} it suffices to show that $\iota_{m+1}$ and $\iota_m$ are trivial. Put $D := \{ b \in B \mid f\cdot b \in A \}$ and note that $D$ is a $\mathbb{Z}_p G$-submodule of $B$, and $f \cdot D$ equals $A \leq f\cdot B$. Thus $\iota_{m+1}$ decomposes into ${H^{m+1}( G, f\cdot D )} \rightarrow H^{m+1}(G, D) \rightarrow H^{m+1}(G, B)$, where the homomorphisms are induced by embeddings. Since $D$ is torsion-free, the $\mathbb{Z}_p G$-module $D$ is isomorphic to $f \cdot D = A$ and the exponent of $H^{m+1}(G, D)$ equals $f = \exp H^{m+1}(G, A)$. It follows that the group homomorphism $H^{m+1}(G, f \cdot D) \rightarrow H^{m+1}(G, D)$ is trivial. Hence $\iota_{m+1}$ is trivial and this leads to the Exact Sequence \eqref{escohoms}. 

Now assume that $A$ is also a submodule of $e \cdot B$. Then $\iota_m$ is the composition $H^m(G, A) \rightarrow H^m(G, e B) \rightarrow H^m(G, B)$. As $H^n(G, e B) \rightarrow H^n(G, B)$ is trivial, the result follows.
\end{proof}

Recall that $\mathbb{Z}_p$ denotes the $p$-adic integers. The following Lemma \ref{lemsescohoms} is a generalized result of Eick \& Leedham-Green \cite[Theorem 18]{Eick2007}. For the proof of Lemma \ref{lemsescohoms} we have chosen another approach enabling us to give sharper bounds. 

\begin{lem} \label{lemsescohoms}
Assume that $B$ is a free $\mathbb{Z}_p$-module of finite rank and $A$ is a submodule of $f \cdot B$ and let $\vartheta_m: H^m(G, B) \rightarrow H^m(G, B/A)$ be as in the Exact Sequence \eqref{escohom}.
\begin{enums}
\item The Exact Sequence \eqref{escohom} gives rise to a split short exact sequence
\begin{align*} 
 0 \rightarrow \im \vartheta_m \longrightarrow H^m(G, B/A) \longrightarrow H^{m+1}(G, A) \rightarrow 0.
\end{align*}
\item If additionally $A \leq e \cdot B$, then the Short Exact Sequence \eqref{sescohom} splits.
\end{enums}
\end{lem}

\begin{proof}
By assumption there exists a $\mathbb{Z}_p G$-submodule $D \leq B$ such that $f \cdot D = A$. Since $G$ is finite and $A = f\cdot D \cong D$ is finitely generated, the groups of cochain maps $C^m(G, D)$ and $C^{m+1}(G, D)$ are finitely generated $\mathbb{Z}_p$-modules. It follows that $H^{m+1}(G, A) \cong H^{m+1}(G, D)$ is finite. 

Let $\vartheta_m$ denote the group homomorphism $Z^m(G, B) \rightarrow Z^m(G, B/A)$ induced by the natural homomorphism $B \twoheadrightarrow B/A$. It is straightforward to show that $B^2(G, B/A)$ is a subgroup of $\im \vartheta_m$. Hence Lemma \ref{lemescohoms} yields that there exists a short exact sequence
\begin{align} \label{sescocycs}
 0 \rightarrow Z^m(G, B) \stackrel{\vartheta_m}{\longrightarrow} Z^m(G, B/A) \longrightarrow H^{m+1}(G, A) \rightarrow 0
\end{align}
and the Short Exact Sequence \eqref{sescohom} splits if and only if this is the case for the Short Exact Sequence \eqref{sescocycs}. Since $H^{m+1}(G, B/A)$ is finite, this is equivalent to the existence of a submodule $K \leq Z^m(G,B/A)$ with $K \cap \im \vartheta_m = 0$ and $K \cong H^{m+1}(G, A)$. Recall that $Z^m(G, D)$ is the kernel of the coboundary map $d^m: C^m(G, D) \rightarrow C^{m+1}(G, D)$, which is a homomorphism of finitely generated $\mathbb{Z}_p$-modules. Let $\nu_p$ denote the $p$-adic valuation. By the existence of the Smith normal form of the coboundary map there are a $\mathbb{Z}_p$-module basis of $C^m(G, D)$, say $a_1, \ldots, a_n$, and corresponding $p$-adic integers, say $b_1, \ldots, b_n$, such that the module $Z^m(G, D/f \cdot D) = \{ \delta \in C^m(G, D) \mid d^m \delta \in C^{m+1}(G, f \cdot D) \}$ equals 
$$
\bigoplus_{\substack{1 \leq i \leq n, \\ \nu_p(b_i) < \nu_p(f)}} \mathbb{Z}_p \cdot f/b_i \cdot a_i \oplus \bigoplus_{\substack{1 \leq i \leq n, \\ \nu_p(b_i) \geq \nu_p(f)}} \mathbb{Z}_p \cdot a_i .
$$
Note that $\vartheta_m( Z^m(G, D) )$ equals $\bigoplus_{1 \leq i \leq n, b_i = 0} \mathbb{Z}_p \cdot a_i$, and hence there exists a $\mathbb{Z}_p$-submodule $K$ of $Z^m(G, D/ f \cdot D)$ such that $Z^m(G, D/f \cdot D)$ is the direct sum $\vartheta_m( Z^m(G, D) ) \oplus K$. By applying Lemma~\ref{lemescohoms} to $Z^m(G, D/ f\cdot D)$ we obtain that $K \cong H^{m+1}(G, f \cdot D) = H^{m+1}(G, A)$. Further, we have $\im \vartheta_m \cap K = \vartheta_m( Z^m(G, D)) \cap K = 0$. The result follows. 
\end{proof}

\subsection{$R$-Homomorphisms}

In this subsection we apply the results of the previous subsection on $H^0(R_{t_0+T_n}, A_n)$, which is isomorphic to $\endo_R A_n$ by Lemma \ref{lemendo}.

For a homomorphism $\epsilon: T \rightarrow T$ with $\epsilon( T_n ) \subseteq T_n$ let $\epsilon_{A_n}: A_n \rightarrow A_n$ be defined as the homomorphism induced by $\epsilon$ and the natural homomorphism $T \twoheadrightarrow A_n$. Further, for a set $M$ of homomorphisms $\epsilon: T \rightarrow T$ with $\epsilon( T_n ) \subseteq T_n$ we write $M_{A_n}$ to denote the set $\{ \epsilon_{A_n} \mid \epsilon \in M\}$.

\begin{thm} \label{thmendop}
Assume that $n \geq \log_p( \exp H^1(R_{t_0}, T_n)) \cdot d$. Then for $\beta \in \aut R$ there exists a complement of $(\endo_R^{\beta} T)_{A_n}$ in $\endo_R^{\beta} A_n$ which is isomorphic to $H^1(R_{t_0}, T_n)$.
\end{thm}

\begin{proof}
Let $\vartheta$ denote the natural homomorphism $T \twoheadrightarrow A_n$, and put $P := R_{t_0}$. First we show that $H^1( P, T_n^{(\beta)}) \cong H^1(P, T_n)$ and $H^0(P, A_n^{(\beta)}) \cong \vartheta( H^0(P, T^{(\beta)}) ) \oplus H^1(P, T_n^{(\beta)})$. Let $\alpha: Z^1( P, T_n^{(\beta)} ) \rightarrow Z^1(P, T_n)$ be defined by $\alpha(\psi)(g) := \psi( g^{\beta^{-1}} )$ for $\psi \in Z^1(P, T_n^{(\beta)})$ and $g \in P$. It is straightforward to check that $\alpha$ is a well-defined group isomorphism and that the image of $B^1(P, T_n^{(\beta)})$ is $B^1(P, T_n)$. It follows that $H^1(P, T_n)$ is isomorphic to $H^1(P, T_n^{(\beta)})$. Hence $\exp H^1(P, T_n)$ equals $\exp H^1(P, T_n^{(\beta)})$, and we have $p^{n/d} \geq \exp H^1(P, T_n^{(\beta)})$ by assumption. In particular $T_n^{(\beta)}$ is a $R$-submodule of $\exp H^1(P, T_n^{(\beta)}) \cdot T^{(\beta)}$ by the uniserial action of $R$ on $T^{(\beta)}$. Then Lemma~\ref{lemsescohoms} yields $H^0(P, A_n^{(\beta)}) \cong \vartheta( H^0(P, T^{(\beta)}) ) \oplus H^1(P, T_n^{(\beta)})$.

Put $Q := R_{t_0 + T_n}$. Next we show that $H^0(Q, A_n^{(\beta)})$ equals $H^0(P, A_n^{(\beta)})$. Since $R$ is finite, there exists a natural number $m$ such that $R_{t_0 + T_{n+md}}$ equals $P = R_{t_0}$. By Lemma \ref{lemendo} we have $\endo_R^{\beta} A_n \cong H^0(Q, A_n^{(\beta)})$ and $\endo_R^{\beta} A_{n+md} \cong H^0(P, A_{n+md}^{(\beta)})$. Note that the $R$-module isomorphism $T \rightarrow T_{md}$, $t \mapsto p^m t$ induces a natural embedding of the group $\endo_R^{\beta} A_{n}$ into $\endo_R^{\beta} A_{n+md}$, whose image is $\hom_R( A_{n+md}, p^m \cdot A_{n+md}^{(\beta)})$. This yields $H^0(Q, A_n^{(\beta)}) \cong H^0(P, p^m \cdot A_{n+md}^{(\beta)})$. As $H^0(P, p^m \cdot A_{n+md}^{(\beta)})$ is canonically isomorphic to $H^0(P, A_{n}^{(\beta)})$, it follows that $H^0(Q, A_n^{(\beta)}) \cong H^0(P, A_n^{(\beta)})$. In particular we have $\endo^{\beta}_R A_n \cong H^0(P, A_n^{(\beta)})$.

Evidently, the image of $(\endo_R T)_{A_n}$ under the isomorphism $\endo_R^{\beta} A_n \rightarrow H^0(P, A_n^{(\beta)})$ of Lemma \ref{lemendo} is $\vartheta( H^0(P, T^{(\beta)}) )$. Recall that $H^0(P, A_n^{(\beta)})$ is isomorphic to the direct sum $\vartheta( H^0(P, T^{(\beta)}) ) \oplus H^1(P, T_n^{(\beta)})$. Hence there exists a complement of $(\endo_R^{\beta} T)_{A_n}$ in $\endo_R^{\beta} A_n$ which is isomorphic to $H^1(P, T_n^{(\beta)}) \cong H^1(P, T_n)$.
\end{proof}

\subsection{Action of the compatible pairs $\Gamma_n$ on $H^2(R, A_n)$} \label{secactioncomp}

In the remaining part of this section we show that there is a $1$-$1$ correspondence between the $\Gamma_n$-orbits of $H^2(R, A_n)$ and the $\Gamma_{n+d}$-orbits of $H^2(R, A_{n+d})$ for every sufficiently large $n$. Put
\begin{align*}
 a := & a(n) := \max \{ \exp H^2(R, T), \exp H^3(R, T_n) \} \ \text{ and } \\
 b := & b(n) := \exp H^1(R_{t_0}, P_n). 
\end{align*}
Since $T_{n+d}$ is torsion-free and $T_{n+d}$ equals $p \cdot T_n$, the cohomology groups $H^3(P, T_{n+d})$ and $H^1(P_{t_0}, T_{n+d})$ are naturally isomorphic to $H^3(P, T_n)$ and $H^1(P_{t_0}, T_n)$, respectively. This yields $a(n+d) = a(n)$ and $b(n+d) = b(n)$. Hence we may assume throughout this subsection that 
\begin{align} \label{assonn}
p^{n/d} \geq \max \{a, b\} \cdot b.
\end{align}
By Theorem \ref{thmendop} there exists a subgroup $E_n \leq \endo_R A_n$ of exponent $b$ with $\endo_R A_n = (\endo_R T)_{A_n} \oplus E_n$. Let $\rho_n$, $\rho$ and $\pi_n$ be the homomorphisms defined by
{\small $$
\xymatrix  @C=1.2pc @R=0.8pc{
  \mathllap{ \varphi \ } \ar@{|->}[rrr]                                 &       &       & \mathrlap{ (1, 1+\varphi) }   & \\
  \mathllap{ E_n } \ar[rrr]^{\rho_n}                                    &       &       & \Gamma_n                      & (\beta, \epsilon_{A_n}) \\
                                                                        &       &       &                               & \\
                                                                        &       &       &                               & \\
  \mathllap{ 1 + \max\{a, b\} \endo_p T } \ar[rrr]^{\rho}               &       &       & \Gamma \ar[uuu]_{\pi_n}       & (\beta, \epsilon) \ar@{|->}[uuu] \\
  \mathllap{\varphi \ } \ar@{|->}[rrr]                                  &       &       & \mathrlap{(1, \varphi)}       &                                     
}
$$
}
\begin{lem} \label{lemmaprhok}
The map $\rho_n$ is a group homomorphism and the image of $\rho_n$ centralizes $\im( \pi_n \circ \rho) \leq \Gamma_n$.
\end{lem}

\begin{proof}
Let $\epsilon$ be an element of $E_n$. The exponent of $E_n$ is $b$ and hence the image of $\epsilon$ is a subgroup of $b^{-1} T_{n}/T_{n}$. Let $m$ be the floor of $n/d$, in particular $n-d < m d \leq n$. As $\max \{ a, b \} \cdot b$ is a $p$-power, we have $p^{n/d} \geq p^m \geq \max\{ a, b \} \cdot b$ by Assumption~\eqref{assonn} and thus $b^{-1} T_n \leq b^{-1} p^m T$. This yields $\im \epsilon \leq b^{-1} T_n/ T_n \leq b^{-1} p^m T/ T_n$. Let $\epsilon'$ be another element of $E_n$. Since multiplication by $p$ induces an $R$-endomorphism $A_n \rightarrow A_n$, 
the composition $\epsilon \circ \epsilon'$ maps from $A_n$ to $(b^{-2} p^m T_{n}+T_{n})/ T_{n}$. Recall that $p^m$ is at least $b^2$. Hence $\epsilon \circ \epsilon'$ is trivial and $(1+\epsilon) \circ (1+\epsilon')$ equals $1+\epsilon +\epsilon'$. Thus $\rho_n$ is a group homomorphism. It follows by similar arguments that $\im \rho_n$ acts trivially on $\im( \pi_n \circ \rho)$.
\end{proof}

\begin{lem} \label{lemnatisocomp}
The groups $\im \rho$ and $\im(\pi_n \circ \rho)$ are normal in $\ct$ and $\cak{n}$, respectively. Further, the group homomorphisms $\rho_n: E_n \rightarrow \cak{n}$ and $\pi_n: \ct \rightarrow \cak{n}$ induce a split short exact sequence
$$
0 \rightarrow E_n \rightarrow \cak{n}/ \im(\pi_n \circ \rho) \rightarrow \Gamma/ \im \rho \rightarrow 1.
$$
\end{lem}

\begin{proof}
Put $c := \max \{a, b\}$ and let $\sigma$ and $\sigma_n$ be the group homomorphisms $\ct \rightarrow \comp(R, T/cT)$, $(\beta, \epsilon) \mapsto (\beta, \epsilon_{T/cT})$ and $\Gamma_n \rightarrow \comp(R, T/cT)$, $(\beta, \epsilon) \mapsto (\beta, \epsilon_{T/cT})$, respectively. Evidently, the image of $\rho$ is the kernel of $\sigma$. Hence $\im \rho$ is normal in $\Gamma$ and $\Gamma/\im \rho$ is isomorphic to $\im \sigma$. By assumption $p^{n/d}/b$ is at least $c$ and thus Theorem \ref{thmendop} and Equations \eqref{eqcomp} and \eqref{eqcompak} yield $\Gamma_n = \langle \im \pi_n, \im \rho_n\rangle$ and $\im \sigma = \im \sigma_n$. Since $\im \rho$ is normal in $\Gamma$, the image $\im (\pi_n \circ \rho)$ is normalized by $\im \pi_n$. Further, by Lemma \ref{lemmaprhok} the group $\im( \pi_n \circ \rho)$ is centralized by $\im \rho_n$ and hence normal in $\Gamma_n =  \langle \im \pi_n, \im \rho_n \rangle$. Note that $\Gamma_n/ \im( \pi_n \circ \rho)$ is isomorphic to $\im \sigma_n$. Now it is straightforward to deduce from $\im \sigma = \im \sigma_n$ that $\rho_n$ and $\pi_n$ induce a split short exact sequence.
\end{proof}

Lemma \ref{lemnatisocomp} enables us to prove Theorem \ref{thmbijorbits}. By Lemma \ref{lemescohoms} the exponent of the second cohomology group $H^2(R, A_n) \cong H^2(R, T) \oplus H^3(R, T_{n})$ is equal to the $p$-power~$a$. As the elements of $\im( \pi_n \circ \rho )$ are of the form $(1, 1 + a \cdot \varphi_{A_n})$ with $\varphi \in \endo_R $, the group $\im( \pi_n \circ \rho)$ acts trivially on $H^2(R, A_n)$. Thus, in order to prove Theorem \ref{thmbijorbits} it suffices to consider the $\cak{n}/ \im (\pi_n \circ \rho)$-orbits of $H^2(R, A_n)$.
 
Recall that $a(n+d) = a(n)$ and $b(n+d) = b(n)$. In particular the results in this subsection also hold for $n+d$. Let $E$ be a set of homomorphisms $T \rightarrow T$ such that $E_{A_n} = E_n$, and put $E_{n+d} := (p \cdot E)_{A_{n+d}}$. Evidently, $E_{n+d}$ intersects trivially with $(\endo_R T)_{A_{n+d}}$ and is isomorphic to $E_n \cong H^1(R_{t_0}, T_n)$. Since $H^1(R_{t_0}, T_n)$ is isomorphic to $H^1(R_{t_0}, T_{n+d})$, it follows by Theorem \ref{thmendop} that $E_{n+d} \oplus (\endo_R T)_{A_{n+d}} = \endo_R A_{n+d}$. Further, by Lemma \ref{lemnatisocomp} the factor groups $\Gamma_n/ \im( \pi_n \circ \rho)$ and $\Gamma_{n+d}/ \im( \pi_{n+d} \circ \rho)$ are naturally isomorphic to $\Gamma/ \im \rho \ltimes E_n$ and $\Gamma/ \im \rho \ltimes E_{n+d}$, respectively.

Hence there exists a unique group isomorphism $\lambda$ such that the following diagram commutes,
$$
\xymatrix@1{
  E_n \ar[r] \ar[d]     & \cak{n}/ \im( \pi_n \circ \rho ) \ar[d]_{\lambda}    & \ct / \im \rho \ar[l] \ar@{=}[d]_{\id} \\
  E_{n+d} \ar[r]        & \cak{n+d}/ \im( \pi_{n+d} \circ \rho )               & \ct / \im \rho \ar[l]
}
$$
where the horizontal arrows are induced by $\rho_n$, $\rho_{n+d}$, $\pi_n$ and $\pi_{n+d}$, and $E_n \rightarrow E_{n+d}$ maps $\epsilon_{A_n}$ to $p \epsilon_{A_{n+d}}$.

Let $m = m(l, S)$ and $(\id \oplus \mu): H^2(R, A_n) \rightarrow H^2(R, A_{n+d})$ be defined as in Theorem~\ref{thmseccohom}.

\begin{thm} \label{thmmumodulhom}
In addition to Assumption \eqref{assonn} we assume that $n$ is at least $m$. For $\tau \in H^2(R, A_n)$ and $g \in \cak{n}/ \im (\pi_n \circ \rho)$ we have $(\id \oplus \mu)( \tau.g ) = (\id \oplus \mu)(\tau).\lambda(g)$. In particular $(\id \oplus \mu)$ induces a $1$-$1$ correspondence from $\cak{n}$-orbits to $\cak{n+d}$-orbits.
\end{thm}

\begin{proof}
It suffices to show that for $\epsilon \in E$ the image of $\tau.(1+\epsilon_{A_n})$ under $(\id \oplus \mu)$ equals $(\id \oplus \mu)(\tau).(1+p \epsilon_{A_{n+d}})$. By construction of $(\id \oplus \mu)$ there exist $\gamma \in Z^2(R, T)$ and $\delta \in C^2(R, T)$ such that $\tau  = \gamma_{A_n} + \delta_{A_n} + B^2(R, A_n)$ and the image of $\delta_{A_n} + B^2(R, A_n)$ under the canonical isomorphism $H^2(R, A_n) \rightarrow H^2(R, T) \oplus H^3(R, T_n)$ lies in the direct summand $H^3(R, T_n)$. In particular
$$
(\id \oplus \mu)(\tau) = \gamma_{A_{n+d}} +  p \delta_{A_{n+d}} + B^2(R, A_{n+d}).
$$
Recall that $a\geq \exp H^3(R, T_n)$ and $b\geq \exp E_n$. Hence we may assume that $\im \delta_{A_n} \leq a^{-1} T_n/ T_n$ and $\im \epsilon_{A_n} \leq b^{-1} T_n/ T_n$. Let $k$ be the floor of $n/d$, in particular $n \geq kd \geq n-d$. As $\max\{a, b\} \cdot b$ is a $p$-power, the inequality $p^{n/d} \geq p^k \geq \max\{a, b\} \cdot b$ holds. Further $T_n$ is a subgroup of $T_{kd} = p^k T$, and thus $\im \epsilon_{A_n} \leq p^k b^{-1} T/T_n \leq a T /T_n$. It follows that the image of $\delta_{A_n}.\epsilon_{A_n}$ is a subgroup of $(p^k a^{-1} b^{-1}T_n + T_n)/T_n$ and hence $\delta_{A_n}.\epsilon_{A_n}$ and $\delta_{A_{n+d}}.(p \epsilon_{A_{n+d}})$ are trivial. This yields
\begin{align*}
 \tau.(1+\epsilon_{A_n}) & = \tau + \gamma_{A_n}.\epsilon_{A_n} + B^2(R, A_n) \ \text{ and } \\
 (\id \oplus \mu)(\tau).(1+p\epsilon_{A_{n+d}}) & = (\id \oplus \mu)(\tau) + \gamma_{A_{n+d}}.(p \epsilon_{A_{n+d}}) + B^2(R, A_{n+d}).
\end{align*}
Since $\im \epsilon_{A_n}$ is a subgroup of $aT/T_n$ and $a$ is at least $\exp H^2(R, T)$, the element $\gamma_{A_n}.\epsilon_{A_n} + B^2(R, A_n)$ lies in the summand of $H^2(R, A_n) \cong H^2(R, T) \oplus H^3(R, T_n)$ which is canonically isomorphic to $H^3(R, T_n)$. Thus the image $(\id \oplus \mu)(\gamma_{A_n}.\epsilon_{A_n} + B^2(R, A_n))$ equals $\gamma_{A_{n+d}}.(p \epsilon_{A_{n+d}}) + B^2(R, A_{n+d})$. The result follows.
\end{proof}

\begin{proof}[Proof of Theorem \ref{thmbijorbits}]
A reformulation of Theorem \ref{thmmumodulhom} yields Theorem \ref{thmbijorbits}
\end{proof}

\section{The graph isomorphism}

In what follows, we construct graph isomorphisms between branches of the shaved coclass tree. This construction goes back to Eick \& Leedham-Green \cite{Eick2007}. The proof in \cite{Eick2007} showing that the graph isomorphism is well-defined is based on the false result \cite[Theorem 23]{Eick2007}. We will give a revised version of the proof. For this purpose, we need the following notation and lemmata.

Let $m$ be as in Theorem~\ref{thmseccohom}, let $n$ denote a natural number, and as in Subsection~\ref{secactioncomp} let $a(n)$ and $b(n)$ be defined as $\max\{ \exp H^2(R, T), \exp H^3(R, T_n) \}$ and $\exp H^1(R_{t_0}, T_n)$, respectively. Recall that $a(n+d) = a(n)$ and $b(n+d) = b(n)$ by the uniserial action of $R$ on $T$. Hence there exists a natural number $v$ such that 
$$
p^v \geq m \text{ and } p^v \geq b(n) \cdot \max\{a(n), b(n)\}
$$
for every natural number $n$.

\begin{lem} \label{lemmucoclass1}
Assume that $n \geq vd$. Let $\tau \in Z^2(R, A_n)$ and $\tau_* \in Z^2(R, A_{n+d})$ be cocycles such that the image of $\tau + B^2(R, A_n)$ under $(\id \oplus \mu)$ equals $\tau_* + B^2(R, A_{n+d})$. Then $E(\tau)$ has coclass $r$ if and only if this is the case for $E(\tau_*)$.
\end{lem}

\begin{proof}
Let $l$ be the nilpotency class of $R$. Note that a group extension $E(\lambda)$ defined by a cocycle $\lambda \in Z^2(R, A_m)$ with $m \in \mathbb{N}_0$ has coclass $r = \cc R$ if and only if $\gamma_l( E(\lambda) ) = A_m$, where $\gamma_l(E(\lambda))$ denotes the $l$-th term of the lower central series of $E(\lambda)$.

Further, recall that $a = a(n) \geq \exp H^3(R, T_n)$. Then by construction of $(\id \oplus \mu)$ there exists $\gamma \in Z^2(R,T)$ and $\epsilon \in C^2(R, T_{n-\log_p a \cdot d})$ such that $\gamma_{A_n} + \epsilon_{A_n} = \tau$ and $\gamma_{A_{n+d}}+ p \cdot \epsilon_{A_{n+d}}$ equals $\tau_*$ modulo $B^2(R, A_{n+d})$. Since $n \geq vd > \log_p a \cdot d$, we have $\epsilon_{A_d} = 0$. It follows that $\gamma_{A_d} = \tau_{A_d} = (\tau_*)_{A_d}$. Then for $\lambda \in \{ \tau, \tau_*\}$ the group extension $E(\lambda)$ has coclass $r$ if and only if $\gamma_l(E(\gamma_{A_d})) = A_d$ holds. The result follows. 
\end{proof}

\begin{lem} \label{lemmucoclass2}
Assume that $n \geq v d$. Then $(\id \oplus \mu)$ induces a bijection between the isomorphism classes of group extensions of $A_n$ by $R$ of coclass $r$ and the isomorphism classes of group extensions of $A_{n+d}$ by $R$ of coclass $r$. 
\end{lem}

\begin{proof}
Let $\gamma_1, \gamma_2 \in Z^2(R, A_n)$ be such that $E(\gamma_1)$ and $E(\gamma_2)$ have coclass $r$ and $E(\gamma_1)$ is isomorphic to $E(\gamma_2)$. Then by Theorem \ref{thmstrongiso} the elements $\gamma_1 + B^2(R, A_n)$ and $\gamma_2 + B^2(R,A_n)$ lie in the same orbit under the action of $\comp(R, A_n)$. By Theorem \ref{thmbijorbits} the map $(\id \oplus \mu)$ induces a bijection between the orbits of $H^2(R,A_n)$ under the action of $\comp(R, A_n)$ and the orbits of $H^2(R, A_{n+d})$  under the action of $\comp(R, A_{n+d})$. The result follows by Lemma \ref{lemmucoclass1}. 
\end{proof}

Lemma \ref{lemmucoclass2} enables us to construct graph isomorphisms between the branches of the shaved coclass tree $\mathcal{T}_k(S)$. For a $p$-group $G$ of coclass $r$ let $\vertex{G}$ denote the vertex in $\mathcal{G}(p,r)$ corresponding to the isomorphism class of $G$. Further, for $n \in \mathbb{N}$ and a cocycle $\gamma \in Z^2(R, A_n)$ defining a group extension $E(\gamma)$ of coclass $r$ let $\vertex{\gamma}$ denote $\vertex{E(\gamma)}$.

Let $l$ be as in Theorem \ref{thmdescendantscctree} and let $\mathcal{B}_i$ denote the $i$-th branch of the shaved coclass tree $\mathcal{T}_k(S)$. In particular $T$ equals $\gamma_{l}(S)$. For $i \geq v d + l$ we write $\nu=\nu(i)$ to denote the map $\mathcal{B}_{i} \rightarrow \mathcal{B}_{i+d}$ induced by $(\id \oplus \mu)$, that is, for a cocycle $\gamma \in Z^2(R,A_n)$ with $n \in \mathbb{N}$ and $\vertex{\gamma} \in \mathcal{B}_{i}$ the vertex $\vertex{\gamma}$ is mapped to $\vertex{\gamma_*}$, where $\gamma_*$ is a representative of $(\id \oplus \mu)(\gamma + B^2(R,A_{n+d}))$.

\begin{thm} \label{thmgraphisoexplicit} 
Let $i \in \mathbb{N}$ be such that $i-l \geq vd$. Then $\nu$ is a well-defined graph isomorphism.
\end{thm}

\begin{proof}
First, note that the image of a vertex under $\nu$ does not depend on the chosen cocycle by Lemma \ref{lemmucoclass2}. 

Let $G$ and $H$ be groups in $\mathcal{B}_{i}$ such that $G$ is an immediate descendant of $H$. Note that $G$ and $H$ are descendants of $S/\gamma_{i}(S)$. Then there are $n \in \mathbb{N}$ and $\tau \in Z^2(R, A_n)$ such that $\vertex{G} = \vertex{\tau}$, $\vertex{H} = \vertex{\tau_{A_{n-1}}}$ and $\vertex{\tau_{A_{i-l}}} = \vertex{S/\gamma_i(S)}$. 

By construction of $(\id \oplus \mu)$ we may assume that there exist $\delta \in Z^2(R, T)$ and $\epsilon \in C^2(R, T_{n- \log_p a(n) \cdot d})$ such that $\tau = \delta_{A_n} + \epsilon_{A_n}$ and for every $j \in \mathbb{N}$ with $n \geq j \geq i-l$ we have
$$
(\id \oplus \mu)( \tau_{A_j} + B^2(R, A_j)) = \delta_{A_{j+d}} + p \epsilon_{A_{j+d}} + B^2(R, A_{j+d}).
$$
If follows that $\nu(V(G))$ is an immediate descendant of $\nu(V(H))$. Since $n \geq vd > \log_p a(n) \cdot d$, the cocycle $\tau_{A_d}$ equals $\delta_{A_d}$. Hence $G$ is a descendant of $E(\delta_{A_d})$ and the coclass of $E(\delta_{A_d})$ equals $r = \cc R = \cc G$ by the uniserial action of $R$ on $T$. It follows that $E(\delta)$ is a pro-$p$-group of coclass $r$. As every coclass tree corresponds to exactly one isomorphism class of an infinite pro-$p$-group of coclass $r$ by definition, the group $E(\delta)$ is isomorphic to $S$. Thus $E(\delta_{A_{i-l}}) \cong S/ \gamma_i(S)$ and $E(\delta_{A_{i-l+d}}) \cong S/\gamma_{i+d}(S)$. This yields that $\nu(V(G))$ is a descendant of $\nu( V(S/\gamma_i(S)) ) = V( S/ \gamma_{i+d}(S) )$ and hence $\im \nu \subseteq \mathcal{B}_{i+d}$. It follows that $\nu$ is a graph isomorphism.

Now, we show that $\nu$ is a bijection. For this purpose let $\nu^{-1}:\mathcal{B}_{i+d} \rightarrow \mathcal{B}_i$ be the graph homomorphism induced by $(\id \oplus \mu)^{-1}$. Similarly to above, we can show that $\nu^{-1}$ is a well-defined graph homomorphism. Evidently $\nu^{-1}$ is an inverse of $\nu$ and thus $\nu$ is a bijection.
\end{proof}

\section{Counterexample to \cite[Theorem 23]{Eick2007}} \label{seccounterexample}

A counterexample to Conjecture \ref{conjtheorem23} \cite[Theorem 23]{Eick2007} is given in this section.  Let $i\in \mathbb{C}$ be the imaginary unit and let $D_8$ denote the dihedral group of order $8$ given by the following group presentation
$$
D_8 := \langle a, b \mid a^2 = 1, b^4 = 1, b^a = b^{-1} \rangle.
$$
Define a uniserial action of $D_8$ on the $\mathbb{Z}_2$-module $\mathbb{Z}_2[i]$ by setting $1.a := 1$, $i.a := -i$ and $1.b := i$, $i.b := i \cdot i = -1$. Then $S := D_8 \ltimes \mathbb{Z}_2[i]$ is a pro-$2$-group of coclass $3$ and the third term of the lower central series of $S$ is $1 \ltimes 2 \cdot \mathbb{Z}_2[i]$. By the uniserial action of $D_8$ on $\mathbb{Z}_2[i]$ it follows that $1 \ltimes 2^k \mathbb{Z}_2[i]$ is the $(1+2k)$-th term of the lower central series for $k \in \mathbb{N}$. Fix $k$, put $T := 1 \ltimes 2^k \mathbb{Z}_2[i] \leq S$ and let $T = T_0 > T_1 > T_2 > \cdots$ denote the lower $S$-central series of $T$. Then for $n \in \mathbb{N}_0$ we consider the action of $\endo_R T/T_n$ on $H^2(R, T/T_n)$, where $\endo_R T/T_n$ acts via the embedding $\endo_R T/T_n \rightarrow \comp(R, T/T_n)$, $\varphi \mapsto (1, \varphi)$ on $H^2(R, T/T_n)$. It can be shown that for every sufficiently large $n$ the group $\endo_R T/T_n$ does not stabilize the direct summand $H^2(R, T)$ of $H^2(R, T/T_n) \cong H^2(R,T) \oplus H^3(R, T_n)$. This contradicts \cite[Theorem 23]{Eick2007}. For further details, we refer to \cite[Appendix B]{Couson2013}.


\def\cprime{$'$}

%
%

\begin{thebibliography}{LGM02}

\bibitem[Cou13]{Couson2013}
Martin Couson.
\newblock {\em {On the character degrees and automorphism groups of finite
  p-groups by coclass}}.
\newblock PhD thesis, 2013.

\bibitem[dS00]{Sautoy2000}
Marcus du~Sautoy.
\newblock Counting {$p$}-groups and nilpotent groups.
\newblock {\em Inst. Hautes \'Etudes Sci. Publ. Math.}, (92):63--112 (2001),
  2000.

\bibitem[Eic06]{Eick2006}
Bettina Eick.
\newblock Automorphism groups of 2-groups.
\newblock {\em J. Algebra}, 300(1):91--101, 2006.

\bibitem[ELG08]{Eick2007}
Bettina Eick and Charles Leedham-Green.
\newblock On the classification of prime-power groups by coclass.
\newblock {\em Bull. Lond. Math. Soc.}, 40(2):274--288, 2008.

\bibitem[LGM02]{Leedham-Green2002}
C.~R. Leedham-Green and S.~McKay.
\newblock {\em The structure of groups of prime power order}, volume~27 of {\em
  London Mathematical Society Monographs. New Series}.
\newblock Oxford University Press, Oxford, 2002.
\newblock Oxford Science Publications.

\end{thebibliography}
%

\end{document}